\documentclass{amsart}
\usepackage{mathptmx, bbm, amscd, amssymb, enumerate, colonequals, mathdots, xcolor, tikz-cd,mathtools, amsthm}
\usepackage{color}
\definecolor{chianti}{rgb}{0.6,0,0}
\definecolor{meretale}{rgb}{0,0,.6}
\definecolor{leaf}{rgb}{0,.35,0}
\usepackage[colorlinks=true, pagebackref, hyperindex, citecolor=meretale, urlcolor=leaf, linkcolor=chianti]{hyperref}

\DeclareMathOperator{\height}{ht}
\DeclareMathOperator{\Tr}{Tr}

\DeclareMathOperator{\fpt}{fpt}
\DeclareMathOperator{\lct}{lct}
\DeclareMathOperator{\ini}{in}

\DeclareMathOperator{\RI}{RI}
\DeclareMathOperator{\codim}{codim}

\newcommand{\fraka}{{\mathfrak{a}}}

\newcommand{\frakm}{{\mathfrak{m}}}
\newcommand{\frakp}{{\mathfrak{p}}}

\newcommand{\bigzero}{\mbox{\normalfont\Large\bfseries 0}}

\def\NN{\mathbb{N}}

\def\Pf{\operatorname{Pf}}

\usepackage{thmtools}
\usepackage{thm-restate}

\newtheorem*{thm*}{Theorem}

\newtheorem{thm}{Theorem}[section]
\newtheorem{lem}[thm]{Lemma}
\newtheorem{cor}[thm]{Corollary}
\newtheorem{prop}[thm]{Proposition}

\theoremstyle{definition}
\newtheorem{defn}[thm]{Definition}

\declaretheorem[
  style=definition,
  title=Example,
  refname={example,examples},
  Refname={Example,Examples},
  sharenumber=thm,
]{exa}

\declaretheorem[
  style=definition,
  title=Remark,
  sharenumber=thm,
]{rmk}

\newtheorem*{question}{Question}

\begin{document}
\title[$F$-purity and FPT as linkage invariants]{$F$-purity and the $F$-pure threshold \\ as invariants of linkage}

\author{Vaibhav Pandey}
\address{Department of Mathematics, Purdue University, 150 N University St., West Lafayette, IN~47907, USA}
\email{pandey94@purdue.edu}

\dedicatory{Dedicated to Professor Bernd Ulrich on the occasion of his seventieth birthday}




\begin{abstract}
The generic link of an unmixed radical ideal is a prime ideal. We show that the squarefreeness of the initial ideal and $F$-purity are, however, not preserved along generic links. On the flip side, for several important cases in liaison theory, including generic height three Gorenstein ideals and the maximal minors of a generic matrix, we show that the squarefreeness of the initial ideal, $F$-purity, and the $F$-pure threshold are each preserved along generic links by identifying a property of unmixed ideals which propagates along generic links. We use this property to establish the $F$-regularity of the generic links of such ideals. Finally, we study the $F$-pure threshold of the generic residual intersections of a complete intersection ideal and answer a related question of Kim--Miller--Niu.      
\end{abstract}

\maketitle
\tableofcontents
\addtocontents{toc}{\protect\setcounter{tocdepth}{1}}

\section{Introduction}
Given a proper, unmixed ideal $I$ in a polynomial ring $R$ over a field $K$, its geometric link is an ideal $J$ such that the ideal $I \cap J$ is generated by a regular sequence, that is, the scheme-theoretic union of the vanishing loci of $I$ and $J$ is a complete intersection. Clearly, the link of the ideal $I$ depends on the choice of a regular sequence in $I$. When this regular sequence is chosen in the `most general' manner, that is, as generic combinations of the generators of the ideal $I$, the link thus obtained, in a polynomial extension of $R$, is called the (first) \textit{generic link} $L_1(I)$ of $I$. Atleast when $R/I$ is a Cohen--Macaulay ring, the generic link specializes to any link of $I$; it is the prototypical link to study since most `good properties' are stable under deformations. Iterating this process, the $n$-th generic link of $I$ 
\[L_n(I)\colonequals L_1(L_{n-1}(I))\]
is the generic link of its $(n-1)$-th generic link, and analogously, models any ideal which is linked to $I$ in $n$ steps. 

Since the beginning of the modern study of linkage, or liaison theory, by Peskine and Szpiro \cite{PS74} and the foundational work by Huneke and Ulrich \cite{Huneke-Ulrich85, Huneke-Ulrich87, Huneke-Ulrich88}, the central research theme has been to understand when two given ideals can be linked in finitely many steps. This calls for studying how the algebro-geometric properties of an ideal change when one passes to its link; in particular, one wishes to understand which properties remain invariant along links. For example, height and the Cohen-Macaulay property are invariants of linkage, and the canonical module of the link is `dual' to that of $R/I$ \cite{PS74}. While it is easily seen that the link of an unmixed, radical ideal may not be radical, by \cite[Proposition 2.6]{Huneke-Ulrich85}, the generic link of an unmixed, radical ideal (even a generically complete intersection ideal) is indeed radical (in fact, prime). It is then natural to search for finer linkage invariants: 

\begin{question}
Let $I$ be an unmixed ideal in a polynomial ring $R$ over a field $K$ and $<$ a term order in $R$. Let $L_n(I)$ denote the $n$-th generic link of $I$ in a polynomial extension $S$ of $R$.
\begin{enumerate}
    \item If the initial ideal $\ini_<(I)$ is squarefree, is the initial ideal of the generic link $L_n(I)$ also squarefree for some term order in $S$ (and for each $n$)?
    \item Assume that $K$ has positive characteristic. If $R/I$ is an $F$-pure ring, is the generic link $S/L_n(I)$ also $F$-pure for each $n$?
\end{enumerate}
\end{question}

Note that if a homogeneous ideal has a squarefree initial ideal with respect to some term order or if it defines an $F$-pure ring, then it is forced to be radical.

We show that, in general, the above question has a \textit{negative} answer: the non-maximal (even sub-maximal) minors of a generic matrix (Example \ref{Example:main}), as well as the rational normal curve (more generally, the maximal minors of a Hankel matrix) (Example \ref{Example:main2}), are counterexamples. In fact, with these examples, we show that even if $R/I$ is $F$-regular, its first generic link may not even be $F$-injective. This observation greatly strengthens the fact that generic linkage does not preserve rational singularities \cite[Corollary 3.4]{Niu}.

The bulk of this paper deals with understanding under what hypothesis are the answers to the above question in the \textit{affirmative}. A major roadblock to our enquiry is that the generators of the generic link are known in very few cases. While a mapping cone construction exists (see \cite{PS74}) to find a free resolution of the link if we know a resolution of the ideal $I$, it is quite cumbersome to carry out in practice. The mapping cone construction gives the explicit generators of the generic link of well-behaved ideals whose resolutions are short, like a perfect ideal of height $2$, a Gorenstein ideal of height $3$ (\cite[\S 4.10]{Kustin-Ulrich-Memoirs}), and a complete intersection ideal (\cite[Example 3.4]{Huneke-Ulrich88}). 

To get around the issue of not knowing the explicit generators of the generic link, for an unmixed, homogeneous ideal $I$ in a polynomial ring $R$, we identify a property $\mathbf{P}$ of the pair $(R,I)$ (see Definition \ref{defn:main}) which propagates along generic links (see Lemma \ref{lemma:main}). Crucially, $\mathbf{P}$ only depends on the choice of a term order in $R$ and on the regular sequence defining the link, and \textit{not} on the generators of the link.

The simple observation that the property $\mathbf{P}$ propagates along generic links is quite powerful: We show that if the pair $(R,I)$ has the property $\mathbf{P}$, then the squarefreeness of the initial ideal, $F$-purity, and the $F$-pure threshold/log canonical threshold are each preserved along the $n$-th generic link of $I$ for each $n$ (Corollary \ref{cor:Main}). This point of view not only greatly extends known results on $F$-singularities to arbitrary generic links and residual intersections, but also provides short, conceptual proofs of known results on the $F$-pure threshold of several families of ideals, as we now highlight:

\subsection*{(1)}
We show that the generic link of a generic height $3$ Gorenstein ideal is strongly $F$-regular in prime characteristic (rational singularities in characteristic $0$) (see Theorem \ref{theorem:PfaffianFregular}). We also determine the $F$-pure threshold/log canonical threshold of the submaximal pfaffians of a generic alternating matrix and its $n$-th generic link (Corollary \ref{cor:FPTPfaffian}). We remark that the generators of the generic link of height three Gorenstein ideals are known (\cite[\S 4.10]{Kustin-Ulrich-Memoirs}, \cite[Theorem 4.7]{Kustin-Ulrich}); however, due to our approach of viewing the generic link via the property $\mathbf{P}$, we do \textit{not} need the generators of the generic link to establish its $F$-regularity (or its $F$-pure threshold, or the squarefreeness of its initial ideal). We also show the $n$-th universal link of a generic height $3$ Gorenstein ideal is $F$-rational in prime characteristic (rational singularities in characteristic $0$) for each $n$ (Corollary \ref{cor:Uni Link F rational}).
\subsection*{(2)}
In \cite[Theorem 6.3]{Pandey-Tarasova}, the authors show that the generic link of the maximal minors of a generic matrix is strongly $F$-regular (note that the generators of the generic link are not known). This result may now be viewed as an offshoot of the fact that the ideal of maximal minors of a generic matrix has the property $\mathbf{P}$ (Proposition \ref{prop:maximal minors}). Remarkably, this perspective allows us to simultaneously recover the $F$-pure threshold/log canonical threshold of the maximal minors of a generic matrix (\cite[Theorem 1.2]{Miller-Singh-Varbaro}, \cite[Theorem 5.6]{Docampo}) and that of the first generic link of the maximal minors of a generic matrix (\cite[Theorem 1]{Kim-Miller-Niu}) by a simpler method. More importantly, it enables us to extend these results to the $n$-th generic link for each $n$ (Remark \ref{remark:RecoverFPT}). 

We show that the ideal of non-maximal minors of a generic matrix typically does \textit{not} have the property $\mathbf{P}$, and due to this, the squarefreeness of the initial ideal and $F$-purity both fail to be transferred to its generic link (Example \ref{Example:main}).
\subsection*{(3)}
The notion of residual intersections, in its modern form, essentially goes back to Artin and Nagata \cite{Artin-Nagata} and has
broad applications in enumerative geometry, intersection theory, the study of Rees rings, and multiplicity theory (\cite{Fulton, Ulrich92, Huneke-Ulrich88}). Residual intersections are a vast generalization of linkage, where, under appropriate technical hypothesis, the regular sequence in consideration is replaced by an ideal of height $s$ with $s \geq \height(I)$ (see \S\ref{subsec: RI} for details). 

We calculate the $F$-pure threshold/log canonical threshold of the generic $s$-residual intersections of a complete intersection ideal having $\mathbf{P}$ (Corollary \ref{cor:FPTofRI}). Our key insight is that the property $\mathbf{P}$ is inherited by the generic $s$-residual intersections of a complete intersection ideal for \textit{each} $s\geq \height(I)$ (Theorem \ref{thm:RIhasP}). In particular, when $s=\height(I)$, we answer \cite[Question 4.5(b)]{Kim-Miller-Niu} of Kim--Miller--Niu by finding a class of complete intersection ideals---those posessing $\mathbf{P}$---for which the $F$-pure threshold is invariant under generic linkage. We note that if the complete intersection ideal does not have $\mathbf{P}$, then its $F$-pure threshold is not preserved under generic linkage (Remark \ref{remark:answerKMN}). 

\section{Background}\label{section:Background}
\subsection{Linkage}
While the notion of linkage holds more generally over Cohen--Macaulay rings (and most of the important results hold true over Gorenstein rings), for the purpose of this paper, and for simplicity of exposition, we restrict our attention to the situation where the ambient ring is polynomial (see the references in this subsection for the general statements).

\begin{defn}\label{definitionLinkage}
    Let $R$ be a polynomial ring, and let $I$ and $J$ be proper $R$-ideals. We say that $I$ and $J$ are \emph{linked} (or $R/I$ and $R/J$ are linked) if there exists an $R$-ideal $\mathfrak{a}$ generated by a regular sequence such that \[J = \mathfrak{a}:I \qquad \text{and} \qquad I = \mathfrak{a}:J,\] and use the notation $I \sim _{\mathfrak{a}} J$. Furthermore we say that the link is \emph{geometric} if we have $\height(I+J) \geq \height(I)+1$ (where $\height(-)$ denotes height). 
\end{defn}

It is clear that the ideal $\mathfrak{a}$ is contained in $I$ and $J$. Note that the associated primes of $I$ and $J$ have the same height, that is, the ideals $I$ and $J$ are unmixed. Further, the heights of the ideals $I$, $J$, and $\mathfrak{a}$ are equal. Moreover, when the link is geometric, the ideal $\mathfrak{a}$ is the intersection of $I$ and $J$.

The linked ideal is `dual' to the given ideal in a sense made precise by the following foundational result of linkage theory: 

\begin{prop}\label{propPS}\cite{PS74}
    Let $R$ be a polynomial ring, $I$ and $J$ be $R$-ideals, and $\mathfrak{a}$ is an ideal generated by a regular sequence such that $I = \mathfrak{a}:J$ and $J = \mathfrak{a}:I$. Suppose that $R/I$ is a Cohen--Macaulay ring. Then 
    \begin{enumerate}[\quad\rm(1)]
        \item $R/J$ is a Cohen--Macaulay ring.
        \item If $R$ is a local ring, $\omega_{R/I} \cong J/\mathfrak{a}$ and $\omega_{R/J} \cong I/\mathfrak{a}$, where $\omega_{R/I}$ (respectively $\omega_{R/J}$) denotes the canonical module of $R/I$ (respectively $R/J$).
        \item If the ideals $I$ and $J$ are geometrically linked, then $\height(I+J) = \height(I)+1$ and $R/(I+J)$ is a Gorenstein ring. 
    \end{enumerate}
\end{prop}

\begin{defn}
    Let $R$ be a polynomial ring and $I$ an unmixed $R$-ideal of height $g>0$. Let $\mathbf{f}\colonequals f_1,\ldots,f_n$ be a generating set of $I$. Let $Y$ be a $g \times n$ matrix of indeterminates, and let $\mathfrak{a}$ be the ideal generated by the entries of the matrix $Y[f_1\dots f_n]^T$ (where $[\quad]^T$ denotes the transpose of the matrix). The ideal \[L_1(I)=L_1(\mathbf{f}) \colonequals \mathfrak{a}R[Y]:IR[Y]\] is a \emph{generic link} of $I$.
\end{defn}

We point out that, in the above definition, the entries of the matrix $Y[f_1\dots f_n]^T$ form an $R[Y]$-regular sequence due to \cite{Hochster73}. A generic link of $I$ is indeed a (geometric) link of $I$:

\begin{prop}\label{prop:PS}\cite{PS74}
    Suppose that $R$ is a polynomial ring and that $I$ is an unmixed ideal of $R$ of height $g$. Let $\mathfrak{a}$ be an $R$-ideal generated by a length $g$ regular sequence which is properly contained in the ideal $I$, and let $J = \mathfrak{a}:I$, then $I \sim _{\mathfrak{a}} J$. 
\end{prop}

While the definition of generic link depends on a choice of a generating set of the ideal, it turns out that any two generic links are equivalent:

\begin{defn}
    Let $(R,I)$ and $(S,J)$ be pairs where $R$ and $S$ are Noetherian rings and $I\subseteq R$, $J \subseteq S$ are ideals. We say $(R,I)$ and $(S,J)$ are \emph{equivalent}, and write $(R,I) \equiv (S,J)$, if there exist finite sets of variables, $X$ over $R$ and $Z$ over $S$, and an isomorphism $\varphi:R[X] \rightarrow S[Z]$ such that $\varphi(IR[X]) = JS[Z]$. 
\end{defn}

\begin{lem}\cite[Proposition 2.4]{Huneke-Ulrich85}\label{lemma:LinksEquiv}
    Let $I$ be an unmixed ideal in a polynomial ring $R$. If $J_1 \subseteq R[X_1]$ and $J_2 \subseteq R[X_2]$ are generic links of $I$, then we have $(R[X_1], J_1) \equiv (R[X_2], J_2)$. 
\end{lem}

Due to the above lemma, we freely use the phrase ``the generic link'' of an ideal in this paper.

\begin{defn} \label{def:Gen/Uni link}
Let $R$ be a polynomial ring and $I$ be an unmixed $R$-ideal of height $g>0$. Let $\mathbf{f}$ be a generating set of $I$. The \textit{$n$-th generic link }of $I$ is the ideal \[L_n(I)=L_1(L_{n-1}(I)) \quad \text{for} \; n>0 \quad \text{with} \quad L_0(I) = I \quad \text{and} \quad L_1(I)=L_1(\mathbf{f}).\]

Assume that $R$ is a regular local ring. The \textit{first universal link $L^1(I)$} of $I$ is obtained similarly as the generic link with respect to invertible indeterminates and the \textit{$n$-th universal link} is obtained iteratively:
\[L^n(I)=L^1(L^{n-1}(I)) \quad \text{for} \; n>0 \quad \text{with} \quad L^0(I) = I \quad \text{and} \quad L^1(I)=L^1(\mathbf{f}).\]
\end{defn}

The \textit{linkage class} of an ideal $I$ in a polynomial ring $R$ is the set of all $R$-ideals which can be obtained from $I$ by a finite number of links.

The point of defining the $n$-th generic link of $I$ (respectively the $n$-th universal link of $I$) is that the 
$n$-th generic link (respectively the $n$-th universal link) is a deformation (respectively, \textit{essentially a deformation}) \footnote{We recall that a ring $A$ is essentially a deformation of another ring $B$ if it is obtained by a finite sequence of deformations and localizations at prime ideals beginning from $B$.} of any $R$-ideal linked to $I$ in $n$ steps (\cite[Proposition 2.14, Theorem 2.17]{Huneke-Ulrich87}) and therefore controls the algebro-geometric properties of any link of $I$ upto deformation. In fact, under mild assumptions, one can descend from a sequence of universal links to a sequence of links in the original ring and still preserve most of the good properties of universal linkage (\cite[Lemma 2.4]{HunekeUlrich-Duke}).
   
\subsection{Frobenius splittings and the $F$-pure threshold}

For a reduced Noetherian ring $R$ of prime characteristic $p>0$, the ($e$-fold) Frobenius endomorphism on $R$ is the map $F^e: R \to R$ with $F^e(r)=r^{p^e}$. To avoid any confusion between the domain and codomain, we instead use the notation $F^e_*(R)$ for the codomain and $F^e_*(r)$ for its elements. So $F^e_*(R)$ is the same ring as $R$ with the $R$-module structure obtained from the Frobenius map:
\[r.F^e_*(s)\colonequals F^e_*(r^{p^e}s).\]

The ring $R$ is \textit{$F$-finite} if $F^e_*(R)$ is a finite $R$-module for some (equivalently, every) $e>0$. A finitely generated algebra $R$ over a field $K$ is $F$-finite if and only if $F_*(K)$ is a finite field extension of $K$---a fairly mild condition. The $F$-finite ring $R$ is \textit{$F$-pure} if the Frobenius endomorphism $F: R \to F_*(R)$ with $F(r)=F_*(r^p)$ is pure, that is, for any $R$-module $M$, the map $F\otimes 1: R\otimes_R M \to F_*(R)\otimes_R M$ is injective. Note that if $R$ is an $F$-pure ring, the Frobenius map $F:R \to F^e_*(R)$ splits as a map of $R$-modules so that the $R$-module $F_*(R)$ admits a free $R$-summand. An $F$-finite ring $R$ is \textit{strongly $F$-regular} if it has a sufficiently large number of Frobenius splittings, made precise as follows: For every element $c \in R^{\circ}$, there exists an integer $q = p^e$ such that the $R$-linear map $R \longrightarrow F^e_*(R)$ sending $1$ to $F^e_*(c)$ splits as a map of $R$-modules. Clearly, a strongly $F$-regular ring is $F$-pure. 

The rings under consideration in this paper are $\NN$-graded; since the various competing notions of $F$-regularity (like weakly $F$-regular, $F$-regular, and strongly $F$-regular) coincide in the graded case due to \cite[Corollary 4.3]{Lyubeznik-Smith}, we make no distinction between them throughout the paper.   

The following result will help us in showing that several ideals considered in this paper define $F$-pure rings.

\begin{lem}\label{corollary:PT}\cite[Lemma 3.1, Corollary 3.3]{Pandey-Tarasova}
Let $R=K[x_1,\ldots,x_n]$ be a polynomial ring over the field $K$ of
characteristic $p>0$ and let $I$ be an unmixed homogeneous $R$-ideal. If $\fraka \subsetneq I$ is an ideal generated by a regular
sequence of length equal to the height of $I$ and $J = \fraka : I$, then
\[\fraka^{[p]}:\fraka \subseteq (I^{[p]}:I) \cap (J^{[p]}:J).\]
In particular, if the ring $S/\fraka$ is $F$-pure, then so are $S/I$ and $S/J$.
\end{lem}

\begin{defn}
Let $R$ be a polynomial ring over a field of characteristic $p>0$ and let $\frakm$ denote its homogeneous maximal ideal. For a homogeneous proper ideal $I$ and integer $e>0$, set
\[ \nu_I(p^e) = \max\{r \in \NN \; | \; I^r \nsubseteq \frakm ^{[p^e]}\}, \]
where $\frakm ^{[p^e]}=(a^{p^e} \;|\;a \in \frakm)$. If $I$ is generated by $N$ elements, it is readily seen that
\[0 \leq \nu_I(p^e) \leq N(p^e-1). \]
Moreover, if $f \in I^r\setminus \frakm ^{[p^e]}$, then $f^p \in I^{pr}\setminus \frakm ^{[p^{e+1}]}$. Thus,
\[\nu_I(p^{e+1}) \geq p\nu_I(p^e).\]
It follows that the sequence of real numbers $\left\{\nu_I(p^e)/p^e \right\}_{e>0}$ is non-decreasing and bounded above; its limit is the \textit{$F$-pure threshold} of $I$, denoted $\fpt(I)$.
\end{defn}

The notion of $F$-pure thresholds is due to Takagi and Watanabe \cite{Takagi-Watanabe}. We point out that the $F$-pure threshold may be defined in a more general setup (see \cite{Mircea-Takagi-Watanabe, Destefani-Luis, Destefani-LuisNagoya}), however the above definition is adequate for this paper. The $F$-pure threshold is the positive characteristic analog of the \textit{log canonical threshold}: a numerical measure of singularity in birational geometry. For simplicity of exposition, let $I$ be a homogeneous ideal in a polynomial ring over the rational numbers. Using ``$I$ modulo $p$" to denote the characteristic $p$ model, one has the inequality
\[\fpt(I \, \text{ modulo } \, p) \leq \lct(I) \quad \text{for all} \; p\gg 0, \]
where $\lct(I)$ denotes the log canonical threshold of $I$. Moreover,
\[ \lim _{p \to \infty} \fpt(I \, \text{ modulo } \, p) = \lct(I).\]
These facts follow from the work of Hara and Yoshida \cite{Hara-Yoshida}; see \cite[Theorems 3.3, 3.4]{Mircea-Takagi-Watanabe}. The following result is likely well-known to the experts:

\begin{lem} \label{lemma:fpt<height}
The $F$-pure threshold of a homogeneous ideal $I$ in a polynomial ring $R$ is bounded above by its height. Furthermore, if the homogeneous ideal is unmixed and radical, then 
\[\fpt(I)=\height(I) \]
implies that the ring $R/I$ is $F$-pure. 
\end{lem}

\begin{proof}
Let $\frakp$ be a prime ideal in the polynomial ring $R=K[\underline{x}]$, where the field $K$ has characteristic $p>0$ and $\frakm$ is the homogeneous maximal ideal of $R$. Then $R_{\frakp}$ is a regular local ring of dimension $\height(\frakp)$. For each $e>0$, the pigeonhole principle gives
\[\frakp _{\frakp}^{\height{\frakp}(p^e-1)+1} \subseteq \frakp _{\frakp}^{[p^e]}.\]
Contracting back to $R$, by the flatness of the Frobenius map, we get
\[\frakp ^{(\height{p}(p^e-1)+1)} \subseteq \frakp ^{[p^e]}.\]
Therefore
\[\frakp ^r \subseteq \frakp ^{(r)} \subseteq \frakp ^{[p^e]} \subseteq \frakm ^{[p^e]} \quad \text{for} \quad r>\height{\frakp}(p^e-1). \]
In particular, let $\frakp$ be a minimal prime of the ideal $I$ of the same height. We get
\[ \fpt(I) \leq \fpt(\frakp) \leq \lim_{e \to \infty} \frac{\height(I)(p^e-1)}{p^e}=\height(I). \]
The last assertion follows from \cite[Theorem 3.11]{Takagi} since the hypothesis
\[\fpt(I)=\height(I) \]
implies that the pair $(R,I^{\height(I)})$ is $F$-pure (see \cite[\S 3]{Takagi} for more on $F$-purity of pairs).
\end{proof}

In the next section, we will calculate the $F$-pure thresholds of certain geometric links. We point out that these calculations are independent of the choice of a generating set of the ideal essentially due to lemma \ref{lemma:LinksEquiv} (alternatively, see \cite[Lemma 2.8]{Kim-Miller-Niu}).  

The next result will be useful in showing that several ideals considered in this paper have a squarefree initial ideal.

\begin{thm} \cite[Theorem 3.13]{Varbaro-Koley} \label{thm:Koley-Varbaro}
    Let $R$ be a polynomial ring over a field. Let $I$ be a radical ideal and $<$ a term order in $R$. Let $g$ be the maximum of the heights of the associated prime ideals of $I$.

    If the initial ideal $\ini_<(I^{(g)})$ contains a squarefree monomial, then $\ini_<(I)$ is a squarefree monomial ideal. 
\end{thm}

\section{The property $\mathbf{P}$}\label{Section:Property P}

\begin{defn}\label{defn:main}
Let $R$ be a polynomial ring over a field and $I$ an unmixed homogeneous $R$-ideal. Let $\mathbf{P}$ be the following property of the pair $(R,I)$: For a fixed term order $<$ in $R$, there exist homogeneous elements $\underline{\alpha} \colonequals \alpha_1, \ldots, \alpha_{\height(I)}$ in the ideal $I$ such that each initial term $\ini_<(\alpha_i)$ is squarefree and each pair of initial terms $\ini_<(\alpha_i), \ini_<(\alpha_j)$ is mutually coprime for $i \neq j$.
\end{defn}

\begin{rmk} \label{remakr:main def}
Since each pair of initial terms $\ini_<(\alpha_i), \ini_<(\alpha_j)$ is mutually coprime, the monomials $\ini_<(\alpha_1), \ldots, \ini_<(\alpha_{\height(I)})$ form an $R$-regular sequence. Therefore the polynomials $\underline{\alpha}= \alpha_1, \ldots, \alpha_{\height(I)}$ also form an $R$-regular sequence (see, for example, \cite[Proposition 1.2.12]{DetBook}). Since $I$ is homogeneous, as is $(\underline{\alpha})$, we note that $(\underline{\alpha}): I$ is also a homogeneous $R$-ideal.      
\end{rmk}

A key insight of this paper is that $\mathbf{P}$ propagates along generic links:

\begin{lem} \label{lemma:main}
Let $I$ be an unmixed homogeneous ideal in a polynomial ring $R=K[\underline{x}]$ over the field $K$. Assume that the pair $(R,I)$ has the property $\mathbf{P}$. 
\begin{enumerate}[\quad\rm(1)]
    \item For any integer $n\geq 0$, let $L_n(I)$ denote the $n$-th generic link of $I$ in a polynomial extension $R[Y]$. The pair $(R[Y],L_n(I))$ has $\mathbf{P}$. 
    \item If the ideal $I$ is not generated by a regular sequence, then $(\underline{\alpha}): I$ is a link of $I$ in the polynomial ring $R$. The pair $(R, (\underline{\alpha}): I)$ has $\mathbf{P}$.
\end{enumerate}
\end{lem}

\begin{proof}
We first prove item $(1)$. Assume that the ideal $I$ has height $g>0$ and proceed by induction on $n$. Since the $n$-th generic link of $I$ is the generic link of the $(n-1)$-th generic link, i.e., \[\ L_n(I) = L_1(L_{n-1}(I)),\] it suffices to prove the assertion for $n=1$. Note that the generic link is a homogeneous ideal for each $n>0$.

Let the elements $\alpha_1, \ldots, \alpha_g$ inside the ideal $I$ be as given by $\mathbf{P}$. Extend it to find generators $\alpha_1, \ldots, \alpha_g,\alpha_{g+1}, \ldots, \alpha_r$ of $I$ and fix this generating set. We proceed to find homogeneous elements $\underline{\beta}\colonequals \beta_1, \ldots, \beta_g$ inside the ideal $L_1(I)$ and a term order $<_1$ in the polynomial ring $R[Y]$, where $Y\colonequals (Y_{i,j})$ is a $g \times r$ matrix of indeterminates, such that the pair $(R[Y], L_1(I))$ has $\mathbf{P}$. Recall that \[L_1(I) = Y[\alpha_1 \cdots \alpha_g \, \alpha_{g+1} \cdots \alpha_r]^TR[Y]: (\alpha_1, \ldots, \alpha_r)R[Y].\]

Define the following variable order in $R[Y]$:
\[ Y_{1,1}>Y_{2,2}>\cdots >Y_{g,g}> \text{ the remaining } Y_{i,j}> x_k, \]
where the order $<$ on the indeterminates $x_k$ is the one given in the property $\mathbf{P}$ and the order on ``the remaining $Y_{i,j}$" is arbitrary. Given a monomial $u\in R[Y]$, we write $u_x\in R$ for the image of $u$ under the map of $R$-algebras from $R[Y]$ to $R$ sending each $Y_{i,j}$ to $1$. Consider the following term order $<_1$ on $R[Y]$: For monomials $u$ and $v$ in $R[Y]$, we define
\[u<_1 v \iff \begin{cases}
    u/u_x \mbox{ is lexicographically smaller than }v/v_x, \\
    u/u_x=v/v_x \mbox{ and }u_x<v_x.
\end{cases}\] 
That is, the term order $<_1$ extends the given term order $<$ of $R$ in a lexicographical manner `along the diagonal'.

Let $\beta_1, \ldots \beta_g$ denote the entries of the matrix $Y[\alpha_1 \cdots \alpha_r]^T$. Then, 
\begin{align*}
\ini_{<_1}(\beta_1) &= \ini_{<_1}(Y_{1,1}\alpha_1 +Y_{1,2}\alpha_2+ \ldots + Y_{1,r}\alpha_r)=\ini_{<_1}(Y_{1,1}\alpha_1) = Y_{1,1}\ini_<(\alpha_1),\\
\ini_{<_1}(\beta_2) &= \ini_{<_1}(Y_{2,1}\alpha_1 +Y_{2,2}\alpha_2+ \ldots +Y_{2,r}\alpha_r)= \ini_{<_1}(Y_{2,2}\alpha_2) = Y_{2,2}\ini_<(\alpha_2),\\
\vdots\\
\ini_{<_1}(\beta_g) &= \ini_{<_1}(Y_{g,1}\alpha_1 +Y_{g,2}\alpha_2 + \ldots + Y_{g,r}\alpha_r)= \ini_{<_1}(Y_{g,g}\alpha_g) = Y_{g,g}\ini_<(\alpha_g).
\end{align*}
Since the pair $(R,I)$ has $\mathbf{P}$, we know that for the elements $\underline{\alpha}=\alpha_1, \ldots, \alpha_g$, each initial term $\ini_<(\alpha_i)$ is squarefree and each pair of initial terms $\ini_<(\alpha_i), \ini_<(\alpha_j)$ is mutually coprime for $i \neq j$. It immediately follows from the above display that the same assertions also hold true for the elements $\underline{\beta}=\beta_1, \ldots, \beta_g$ of $R[Y]$.

We now prove item $(2)$. Since $I$ is not generated by a regular sequence, the proper (homogeneous) $R$-ideal $\underline{\alpha}:I$ is a link of the unmixed ideal $I$ by Proposition \ref{prop:PS}. The pair $(R,(\underline{\alpha}): I)$ has $\mathbf{P}$ because 
\[\underline{\alpha} \subseteq (\underline{\alpha}): I. \qedhere \]
\end{proof}

\begin{rmk}
Note that the above lemma holds true without requiring the ideal $I$ to be unmixed. Furthermore, it also holds true if the phrase ``each initial term $\ini_<(\alpha_i)$ is squarefree" is deleted from the property $\mathbf{P}$ or if the sequence of elements $\underline{\alpha}$ is shorter. However, the applications that we have in mind require $\mathbf{P}$ to be as stated.    
\end{rmk}

In the remainder of this section, we show that if the pair $(R,I)$ has $\mathbf{P}$, then important properties of the ideal $I$ like the initial ideal being squarefree, $F$-purity, and the $F$-pure threshold remain invariant under linkage. In fact, we get that the $F$-pure threshold (or the log canonical threshold) of each generic link of $I$, and hence that of $I$ itself, attains its maximal value. That is, the $F$-pure threshold of the $n$-th generic link is equal to the height---an invariant of linkage! 

\begin{thm} \label{theorem:main}
Let $R$ be a polynomial ring over a field $K$. Let $I$ be an unmixed homogeneous $R$-ideal. Assume that the pair $(R,I)$ has the property $\mathbf{P}$.  
\begin{enumerate}[\quad\rm(1)]
     \item If the field $K$ has positive characteristic, then
    \[\fpt(I) = \height(I).\]
    Furthermore, $R/I$ is an $F$-pure ring. 

    \item If the field $K$ has characteristic zero, then 
    \[\lct(I) = \height(I).\]
    Furthermore, $R/I$ has log canonical singularities.  

    \item The initial ideal of $I$ is squarefree.
\end{enumerate}

\begin{proof}
Note that item $(2)$ immediately follows from item $(1)$ since $F$-pure rings in characteristic $p>0$ are log canonical in characteristic zero (\cite[Theorem 3.9]{Hara-Watanabe}) and the log canonical threshold of an ideal is the limit of the $F$-pure thresholds of its characteristic $p>0$ models. In view of this, assume that $K$ has characteristic $p>0$. We begin by proving item $(1)$. Assume throughout that the ideal $I$ has height $g>0$ and let $\underline{\alpha} \colonequals \alpha_1, \ldots, \alpha_{g}$ be as stated in property $\mathbf{P}$.  

Since $\underline{(\alpha)} \subseteq I$ and, by Lemma \ref{lemma:fpt<height}, the $F$-pure threshold of an ideal is bounded above by its height, we have
\[\fpt (\underline{(\alpha)}) \leq \fpt(I) \leq g.\]
We claim that $\fpt (\underline{(\alpha)})=g$. This would give us
\[g = \fpt (\underline{(\alpha)}) \leq \fpt(I) \leq g,\]
so that we have equality throughout. 

We now show that $\fpt (\underline{(\alpha)})=g$. The initial term of \[f \colonequals \alpha_1\ldots \alpha_g \in \underline{(\alpha)}^g\]
is squarefree as the pair $(R,I)$ has $\mathbf{P}$. For any $e>0$, $\frakm ^{[p^e]}$ is a monomial ideal. So, we have
\begin{equation}\label{Equation 1}
f^{p^e-1} = (\alpha_1\ldots \alpha_g)^{p^e-1} \in ((\underline{\alpha}^{[p^e]}):\underline{(\alpha)}) \smallsetminus \frakm ^{[p^e]}.   
\end{equation}
Hence
\[ (\underline{\alpha})^{g(p^e-1)} \nsubseteq \frakm ^{[p^e]}. \quad \text{However} \quad (\underline{\alpha})^{g(p^e-1)+1} \subseteq \frakm ^{[p^e]} \]
by the pigeonhole principle. So we conclude that 
\[ \nu_{({\underline{\alpha}})}(p^e)= \max\{r \; | \; (\underline{\alpha})^r \nsubseteq \frakm^{[p^e]} \} = g(p^e-1) \]
for each $e>0$, so that
\[ \fpt(\underline{\alpha})=g, \]
as required; we thus have  
\[\fpt(I) = \height(I).\]
The $F$-purity of the ring $R/I$ follows from Lemma \ref{lemma:fpt<height}; alternatively, it also follows from Equation \ref{Equation 1} in view of Lemma \ref{corollary:PT}. 
  
We now prove item $(3)$. Since the $R/I$ is $F$-pure, we get that $I$ is a radical ideal. Consider the polynomial \[f \colonequals \alpha_1\ldots \alpha_g \in I^g.\]  Then the initial term
\[\ini_<(f) = \prod_{i=1}^g \ini_<(\alpha_i) \in \ini_<((I^g)\subseteq \ini_<((I^{(g)})\]
and is squarefree, as noted above. All the requirements of Theorem \ref{thm:Koley-Varbaro} are met; we conclude that $I$ has a squarefree initial ideal.
\end{proof}

\begin{cor}\label{cor:Main}
Let $R$ be a polynomial ring over a field $K$. Fix an integer $n\geq 0$ and let $I$ be an unmixed homogeneous $R$-ideal such that the pair $(R,I)$ has $\mathbf{P}$.    
\begin{enumerate}[\quad\rm(1)]
    \item If the field $K$ has positive characteristic, then
    \[\fpt(L_n(I)) = \height(I).\]
    Furthermore, the link $L_n(I)$ defines an $F$-pure ring (in a polynomial extension of $R$).

    \item If the field $K$ has characteristic zero, then
    \[\lct(L_n(I)) = \height(I).\]
    Furthermore, the link $L_n(I)$ defines a log canonical singularity (in a polynomial extension of $R$).
    \item The initial ideal of $L_n(I)$ is squarefree.
\end{enumerate}

\begin{proof}
The generic link $L_n(I)$ is an ideal of a polynomial extension $R[Y]$ of $R$. The ideals $IR[Y]$ and $L_n(I)$ are linked in $R[Y]$ so $\height (L_n(I))=\height(I)$. Crucially, since the pair $(R,I)$ has the property $\mathbf{P}$ (and $L_0(I) = I$), by Lemma \ref{lemma:main}, the pair $(R[Y],L_n(I))$ also has $\mathbf{P}$. Each of the assertions follow immediately from Theorem \ref{theorem:main}.   
\end{proof}
\end{cor}

\normalfont{We now specialize to a link of $I$ in the ambient polynomial ring for which the conclusions of Theorem \ref{theorem:main} hold:}   

\begin{cor}\label{cor:main}
Let $R$ be a polynomial ring over a field $K$ and $I$ an unmixed homogeneous $R$-ideal which is not generated by a regular sequence. Assume that the pair $(R,I)$ has $\mathbf{P}$. Then the ideal $(\underline{\alpha}): I$ is a link of $I$. Moreover,

\begin{enumerate}[\quad\rm(1)]
    \item If the field $K$ has positive characteristic, the link $(\underline{\alpha}): I$ defines an $F$-pure ring and
    \[\fpt((\underline{\alpha}): I) = \height(I).\]

    \item If the field $K$ has characteristic zero, the link $(\underline{\alpha}): I$ defines a log-terminal singularity and
    \[\lct((\underline{\alpha}): I) = \height(I).\]

    \item The initial ideal of $(\underline{\alpha}): I$ is squarefree.
\end{enumerate}
\end{cor}

\begin{proof}
 Since $I$ is not generated by a regular sequence, $(\underline{\alpha}): I$ is a proper $R$-ideal; it is a link of $I$ by Remark \ref{remakr:main def}. By the second assertion of Lemma \ref{lemma:main}, the pair $(R, (\underline{\alpha}): I)$ has $\mathbf{P}$. The arguments for each of the assertions are precisely the same as those of Theorem \ref{theorem:main}.    
\end{proof}
\end{thm}

In \cite[Question 4.5 (a)]{Kim-Miller-Niu}, Kim--Miller--Niu ask if the equality of the log canonical thresholds of a variety $X$ and that of its generic link implies the equality of the log canonical threshold of $X$ and those of the higher generic links of $X$. In view of Corollary \ref{cor:Main}, the examples listed below provide evidence in favor of an affirmative answer: For each pair $(R,I)$ in the examples below, the log canonical threshold of the variety and that of its $n$-th generic link are equal for each $n>0$.

\begin{rmk}
In each of the examples discussed in the next section, the elements $\underline{\alpha}$ discussed in the property $\mathbf{P}$ form a part of a minimal generating set of the ideal of interest. Hence the $n$-th (homogeneous) generic link that we get in each case is actually the $n$-th \textit{minimal} (homogeneous) generic link. Minimal linkage is generally preferable over arbitrary linkage: The minimal link of an almost complete intersection ideal defines a Gorenstein ring. For other favorable properties of minimal links, see \cite[Remark 2.7]{Huneke-Ulrich87}. 
\end{rmk}

\section{Consequences of the Property $\mathbf{P}$}
\addtocontents{toc}{\protect\setcounter{tocdepth}{2}}
\subsection{$F$-regularity of the generic link of generic height $3$ Gorenstein ideals}\label{Section:Submax Pfaffians}

Let $I$ be a height $3$ ideal in a polynomial ring $R$ such that $R/I$ is a Gorenstein ring. By the Buchsbaum--Eisenbud structure theorem \cite{Buchsbaum-Eisenbud}, the generators of $I$ are given by the size $2n$ pfaffians of a $(2n+1)\times (2n+1)$ skew-symmetric matrix. If the entries of this matrix $X$ are indeterminates, then $I\colonequals\Pf_{2n}(X)$ is said to be a `generic' height $3$ Gorenstein ideal in the ring $R=K[X]$. Note that any other such ideal can be obtained by specializing the entries of $X$. 

\begin{prop}\label{prop:pfaffian}
The pair $(K[X],\Pf_{2n}(X))$ has $\mathbf{P}$. 
\end{prop}

\begin{proof}
 Let
 \[\underline{\alpha}\colonequals [1,2,\ldots,2n], [1,2,\ldots,n,n+2,\ldots,2n+1],[2,3,\ldots,2n+1] \]
 denote the pfaffians of $X$ with the given rows and columns and fix the lexicographical order in $K[X]$ induced by the variable order
 \[x_{1,2n+1}>x_{1,2n}> \cdots >x_{1,2}>x_{2,2n+1}>\cdots > x_{2,3}>x_{3,2n+1}>\cdots >x_{2n,2n+1}  \]
which moves from right to left along any row of the matrix $X$ and then vertically downwards. The initial terms
\begin{align*}
\ini_<([1,2,\ldots,2n]) &= \prod_{i=1}^n x_{i,2n+1-i}, \\
\ini_<([1,2,\ldots, n,n+2,\ldots, 2n+1]) &= \prod_{i=1}^n x_{i,2n+2-i}, \\
\ini_<([2,3,\ldots, 2n+1]) &= \prod_{i=2}^{n+1} x_{i,2n+3-i}
\end{align*}
are precisely as required for the property $\mathbf{P}$ to hold. Since the ideal $\Pf _{2n}(X)$ has height $3$, we are done.
\end{proof}

\begin{cor}\label{cor:FPTPfaffian}
 Let $X$ be a $(2n+1)\times (2n+1)$ skew-symmetric matrix of indeterminates and $I=\Pf_{2n}(X)$ be the ideal of submaximal pfaffians in $K[X]$; set $n\geq 0$. Then
 \[\fpt(L_n(I)) = 3.\] In particular, $\fpt(I) = 3$.
\end{cor}

\begin{proof}
This is immediate from Proposition \ref{prop:pfaffian} in view of Corollary \ref{cor:Main}.    
\end{proof}

Proposition \ref{prop:pfaffian} has a remarkable consequence: The generic link of a generic height $3$ Gorenstein ideal defines an $F$-regular ring.


\begin{thm}\label{theorem:PfaffianFregular}
Let $X\colonequals (x_{i,j})$ be a $(2n+1)\times(2n+1)$ skew-symmetric matrix of indeterminates, $K$ a field, and $R\colonequals K[X]$. Let $\Pf_{2n}(X)$ denote the ideal generated by the size $2n$ pfaffians of $X$.
\begin{enumerate}[\quad\rm(1)]
    \item If $K$ is an $F$-finite field of characteristic $p>0$, the generic link of $R/\Pf_{2n}(X)$ is $F$-regular.
    \item If $K$ has characteristic $0$, the generic link of $R/\Pf_{2n}(X)$ has rational singularities. 
\end{enumerate}

\begin{proof}
Assertion $(2)$ follows from $(1)$ since $F$-regular rings are $F$-rational and $F$-rational rings have rational singularities by \cite[Theorem 4.3]{Smith}. We therefore concentrate on the case where the characteristic of $K$ is positive.    

Let $Y\colonequals (Y_{i,j})$ be a generic $3 \times (2n+1)$ matrix of indeterminates and $S\colonequals R[Y]$. The proof will proceed by induction on $n$. Note that if $n=1$, the ideal $\Pf_{2n}(X) = \Pf_2(X)$ is the homogeneous maximal ideal of a polynomial ring in $3$ variables. So the $F$-regularity of the generic link is a special case of \cite[Theorem 5.1]{Pandey-Tarasova}. For the remainder of the proof, assume $n\geq 2$. Let $\Delta_1,\ldots,\Delta_{2n+1}$ denote the pfaffians of $X$ with 
\[\Delta_{2n}=[1,3,\ldots,2n+1] \quad \text{and} \quad \Delta_{2n+1}=[2,3,\ldots,2n+1].\]
We have,  
\[J = Y[\Delta_1,\ldots,\Delta_{2n+1}]^TS: (\Delta_1,\ldots,\Delta_{2n+1})S,\]
and $S/J$ is the generic link of $R/\Pf_{2n}(X)$. 

We first show that the ring $(S/J)_{x_{1,2}}$ is $F$-regular. To prove the inductive step $n>1$, we show that the localization $(S/J)_{x_{1,2}}$ is a faithfully flat extension of the generic link $S'/J'$ of the size $2n-2$ pfaffians of a generic skew-symmetric $(2n-1)\times (2n-1)$ matrix $X'$, with $S'=K[X']$. 

Let $\alpha_i$ denote the $i$-th entry of the matrix $Y[\Delta_1,\ldots,\Delta_{2n+1}]^T$ for $1\leq i \leq 2n+1$; set $\alpha_i'\colonequals \alpha_i/x_{1,2}$ for $i=1,2,3$ and $\Delta_j' \colonequals \Delta_j/x_{1,2}$ for $j=1,\ldots, 2n+1$. Consider the linear change of coordinates,
\[x'_{i-2,j-2} \colonequals x_{i,j} + \frac{x_{1,j}x_{2,i}-x_{1,i}x_{2,j}}{x_{1,2}} \qquad i,j=3,\ldots,2n+1 \]
and set $X'\colonequals (x'_{i,j})$, a generic skew-symmetric $(2n-1)\times (2n-1)$ matrix. By 
\cite[Lemma 1.3]{Barile}, there exists an invertible matrix $E$, with entries in $R_{x_{1,2}}$, such that

\[E^TXE = \left(\begin{array}{@{}c|c@{}}
  \begin{matrix}
  0 & 1 \\
  -1 & 0
  \end{matrix}
  & \bigzero \\
\hline
  \bigzero &
  \begin{matrix}
  0 & x'_{1,2} & \cdots & x'_{1,2n-1} \\
  -x'_{1,2} & 0 & \cdots & x'_{2,2n-1} \\
  \vdots & \vdots & \ddots & \vdots \\
  -x'_{1,2n-1} & -x'_{2,2n-1} & \cdots & 0
  \end{matrix}
\end{array}\right) . \]

A routine, albeit tedious, calculation gives 
\[\alpha_i' = Y_{i,1}'\Delta_1' + \ldots + Y_{i,2n-1}'\Delta_{2n-1}' \qquad i=1,2,3, \]
where $\Delta_1',\ldots, \Delta_{2n-1}'$ are the size $2n-2$ pfaffians of $X'$ and the elements $Y'_{i,j}$ are algebraically independent over the field $K$; note that
\[\Pf_{2n}(X) = \Pf_{2n}(E^TXE) = \Pf_{2n-2}(X'). \]
Let $Y'\colonequals (Y'_{i,j})$ be a $3 \times (2n-1)$ (generic) matrix; set $R' \colonequals K[X']$, $S' \colonequals R[Y']$, and 
\[J' \colonequals Y'[\Delta_1',\ldots, \Delta_{2n-1}']S':(\Delta_1',\ldots, \Delta_{2n-1}')S'. \]
Then we get an isomorphism
\[ (S/J)_{x_{1,2}} \cong S'/J'[x_{1,3},\ldots ,x_{1,2n+1},x_{2,3}, \ldots, x_{2,2n+1},Y_{i,j}|i=1,2,3; \; j=2n,2n+1], \]
as desired (where $\Delta_i \mapsto \Delta_i'$ for $i=1,\ldots, 2n-1$). Therefore, the ring $(S/J)_{x_{1.2}}$ is $F$-regular by induction.

We now construct an $S/J$-linear splitting of the map \[ S/J \to F_*(S/J) \qquad \text{where} \qquad 1 \mapsto F_*(x_{1,2})\] in order to establish the $F$-regularity of the generic link $S/J$. After reindexing, let
\[\Delta_1 = [1,\ldots,2n], \; \Delta_2 = [1,\ldots,n,n+2,\ldots,2n+1], \; \text{and}\; \Delta_3 = [2,\ldots,2n+1] \]
denote the pfaffians of $X$ with the given rows and columns and let $\beta_1 , \beta_2, \beta_3$ denote the first three entries of the matrix $Y[\Delta_1,\ldots,\Delta_{2n+1}]^T$. By Lemma \ref{lemma:main} and Proposition \ref{prop:pfaffian}, we get that the initial term of the polynomial
\[f \colonequals \beta_1 \beta_2 \beta_3 \]
is squarefree with respect to the term order $<_1$ in Lemma \ref{lemma:main} (constructed by extending the term order $<$ in Proposition \ref{prop:pfaffian}). Furthermore, we have,
\[f^{p-1} \in J^{[p]}:J\]
by Lemma \ref{corollary:PT}. Note that $x_{1,2}$ does not divide the initial term $\ini_{<_1}(f)$. In view of this, consider the polynomial
\[ g \colonequals \prod _{\substack {Y_{i,j},x_{k,l} \notin \ini_{<_1}(f) \\ (k,l)\neq(1,2)}}  Y_{i,j}x_{k,l}f\]
and note that $g^{p-1} \in J^{[p]}:J$. The Frobenius trace map
\[\Tr(F_*(x_{1,2}^{p-2}g^{p-1})): F_*(S/J) \to S/J \quad \text{sends} \quad F_*(x_{1,2})\mapsto 1 \]
to give the required splitting. The $F$-regularity of the generic link $S/J$ now follows from \cite[Theorem 5.9(a)]{Hochster-Huneke94a}.
\end{proof}
\end{thm}

\begin{rmk}
The generators of the generic link of height $3$ Gorenstein ideals are known: \cite[\S 4.10]{Kustin-Ulrich-Memoirs} gives the generators of the generic link as nested pfaffians of varying sizes of an `almost skew-symmetric' matrix (see \cite[Theorem 4.7]{Kustin-Ulrich} for an easier proof). We emphasize that due to our approach of viewing the generic link via the property $\mathbf{P}$, we did \emph{not} need its generators to establish its $F$-regularity (or its $F$-pure threshold, or the squarefreeness of its initial ideal). 
\end{rmk}

\begin{cor}\label{cor:Uni Link F rational}
Let $I$ be a generic height $3$ Gorenstein ideal of $R=K[\underline{x}]$ over an $F$-finite field $K$ of characteristic $p>0$ and $L^k(I)$ the $k$-th universal link of $I$ in a transcendental extension $S$ of $R$. Then $S/L^k(I)$ is an $F$-rational ring for each $k\geq 0$.      
\end{cor}
\begin{proof}
 Due to \cite[Corollary 2.15]{Huneke-Ulrich87}, the even universal links $S/L^{2k}(I)$ (respectively the odd universal links $S/L^{2k+1}(I)$) are essentially a deformation of $R/I$ (respectively that of $S/L^1(I)$). That is, the universal links are obtained by a finite sequence of deformations and localizations at prime ideals beginning from $R/I$ or $S/L^1(I)$, in the respective cases. Note that generic pfaffian rings are $F$-regular and the generic links of generic height $3$ Gorenstein ideals are $F$-regular by Theorem \ref{theorem:PfaffianFregular}. Since the universal link is a localization of the corresponding generic link at a prime ideal, the assertion follows as $F$-regular rings are $F$-rational and $F$-rationality deforms (and is a local property) due to \cite[Theorem 4.2 (h)]{Hochster-Huneke94a}.       
\end{proof}

\subsection{Minors of a generic matrix: A curious dichotomy}\label{Section: Minors}
While the generic link of the maximal minors of a generic matrix is $F$-regular \cite[Theorem 5.6]{Pandey-Tarasova}, in this section, we show that the generic link of the non-maximal (even submaximal) minors is usually not even $F$-injective; the point is that they do not posses the property $\mathbf{P}$. Let $X\colonequals (x_{ij})$ be an $m \times n$ matrix of indeterminates for positive integers $m \leq n$ and $R\colonequals K[X]$. Let $I_t(X)$ denote the $R$-ideal generated by the size $t$-minors of the matrix $X$.

\begin{prop} \label{prop:maximal minors}
The pair $(K[X]), I_m(X))$ has $\mathbf{P}$.    
\end{prop}

\begin{proof}
Let \[\underline{\alpha}\colonequals [1,m], [2,m+1], \ldots, [n-m+1,n]\] denote the size $m$ minors of $X$ with the given adjacent columns and fix the lexicographical order in $K[X]$ induced by the variable order
\[x_{1,1} > x_{1,2} > \dots > x_{1,n} > x_{2,1} > \dots > x_{2,n} > \dots > x_{m,n}.\]
For $1 \leq i \leq n-m+1$, by determinant expansion along the first row of $X$, we get 
\[\ini_<\big([i,m+i-1]\big) = x_{1,i}x_{2,i+1}\dots x_{m,m+i-1},\] which is the product of the variables in the main diagonal of the minor $[i,m+i-1]$. Since the determinantal ideal $I_m(X)$ has height $n-m+1$, we are done.
\end{proof}
\begin{rmk}\label{remark:RecoverFPT}
In view of Corollary \ref{cor:Main}, Proposition \ref{prop:maximal minors} simultaneously recovers the $F$-pure threshold (respectively log canonical threshold) of the ideal of maximal minors calculated in \cite[Theorem 1.2]{Miller-Singh-Varbaro} (respectively \cite[Theorem 5.6]{Docampo}, \cite{Johnson}) and the log canonical threshold of the generic links of maximal minors calculated in \cite[Theorem 1]{Kim-Miller-Niu} by a simpler method and extends these calculations to the $n$-th generic link for each $n>0$.   
\end{rmk}

\begin{lem}\label{lemma:a-invariant}
Let $R = K[x_1,\ldots, x_n]$ and $I$ be an unmixed homogeneous $R$-ideal of height $g$ for which each generator has degree $d$. Then the (top) $a$-invariant of the first universal link of $I$ is $d(g-1)-n$.     
\end{lem}

\begin{proof}
 Let $I=(f_1,\ldots,f_k)\subset R$, $Y:=(Y_{i,j})$ be a $g\times k$ matrix, and $S:=R_{(x_1,\ldots,x_n)}(Y_{i,j})$. Let
 \[L^1(I) \colonequals Y[f_1,\ldots, f_k]^TS:(f_1,\ldots, f_k)S = (\underline{\alpha}):IS \]
 denote the (first) universal link of $I$. Note that, since the ideal $I$ is generated in a single degree, the factor ring $S/L^1(I)$ is the localization of a graded ring at its homogeneous maximal ideal. Hence its top local cohomology module
 \[H^{n-g}_{(x_1,\ldots,x_n)} (S/L^1(I))\]
is a $\mathbb{Z}$-graded $S/L^1(I)$-module in a natural manner. By \cite[Lemma 2.3]{Kustin-Ulrich}, the graded canonical module of the universal link is
 \[\omega_{S/L^1(I)} \cong IS/(\underline{\alpha})[-n+ \sum_{i=1}^g \deg(\alpha_i)]. \]

Since the indeterminates $Y_{ij}$ are units in $S$ and so have degree zero, it follows that the graded canonical module $\omega_{S/L^1(I)}$ is generated in degree 
 \[d-(-n+gd) = n-d(g-1).\]
Therefore the $a$-invariant
 \[a(S/L^1(I)) = d(g-1)-n \]
 is as claimed.
\end{proof}

\begin{exa}\label{Example:main}
Let $X$ be a $3 \times 7$ matrix of indeterminates and $R=K[X]$. The pair $(R,I_2(X))$ clearly does not have the property $\mathbf{P}$: The ideal $I_2(X)$ has height $12$ and does not contain any sequence of elements $\alpha_1, \ldots, \alpha_{12}$ as required for $\mathbf{P}$ to hold since that would require atleast $24$ indeterminates. 

By Lemma \ref{lemma:a-invariant}, the $a$-invariant of the (first) universal link is \[a(R(Y)/L^1(I_2(X))) = 2(12-1)-21=1>0.\] Since a graded $F$-injective ring has a non-positive $a$-invariant, it follows that the universal link $L^1(I_2(X))$ does not define an $F$-injective ring. As the universal link is a localization of the generic link at a prime ideal and since $F$-injectivity is a local property \cite[Proposition 3.3]{Datta-Murayama}, it also follows that the generic link $L_1(I_2(X))$ is \emph{not} $F$-injective. Therefore it is not $F$-pure. This observation stands in sharp contrast to the recent result \cite[Theorem 5.6]{Pandey-Tarasova} which says that the generic link of the maximal minors of a generic matrix is $F$-regular. 

Since a squarefree initial ideal implies $F$-injectivity (see \cite[Corollary 4.11(2)]{Varbaro-Koley}), it follows that the generic link $L_1(I_2(X))$ does \emph{not} have a squarefree initial ideal. However, the generic determinantal ideal of minors of any size has a squarefree initial ideal since the minors form a Gr\"obner basis with respect to a diagonal term order by \cite{Sturmfels}. We conclude that (generic) linkage does \emph{not} preserve the squarefreeness of the initial ideal.  

In addition, since generic determinantal rings (of minors of any size) are $F$-regular (\cite[\S 7]{Hochster-Huneke94b}) and hence have rational singularities (\cite[Theorem 4.3]{Smith}), we recover the fact that (generic) linkage does \textit{not} preserve rational singularities (see \cite[Corollary 3.4]{Niu}). Finally, note that by \cite[Theorem 1]{Kim-Miller-Niu} (see also \cite[Theorem 5.6]{Docampo}),
\[\fpt(L_1(I_2(X))) = 10.5<12=\height(I_2(X)).\] 
\end{exa}

We now show that the property $\mathbf{P}$ does not hold for the (maximal) minors of a Hankel matrix, which are linear specializations of the maximal minors of a generic matrix:

\begin{exa}\label{Example:main2}
 Let $H$ be a $t \times n$ Hankel matrix with $t \geq 2$ and $n \geq 2t+2$. Let $I_t(H)$ denote the ideal generated by the size $t$ minors of $H$ in the polynomial ring $R=K[H]$. By Lemma \ref{lemma:a-invariant}, the $a$-invariant of the universal link
 \[a(R(Y)/L^1(I_t(H))) = t(n-t)-(n+t-1) = (n-1)(t-1)-t^2 \geq t^2-t-1>0. \]
 By the same arguments as in Example \ref{Example:main}, we conclude that the generic link $L_1(I_t(H))$ does not define an $F$-injective ring and does not have a squarefree initial ideal. However, the Hankel determinantal ring $R/I_t(H)$ is $F$-rational if the field $K$ has characteristic $p$ with $p \geq t$, and therefore has rational singularities (see \cite[Theorem 2.1]{CMSV}). Furthermore, the ideal $I_t(H)$ has a squarefree initial ideal as its generators form a Gr\"obner basis with respect to the usual lexicographical ordering on the variables by \cite[Proposition 3.4]{Conca}.

 Note that the case $t=2$ gives us that the generic link of the rational normal curve is not $F$-injective and does not have a squarefree initial ideal for any $n \geq 6$.
\end{exa}

\section{Generic residual intersections of a complete intersection inherit $\mathbf{P}$}\label{Section: RI}

In this section, we show that the property $\mathbf{P}$ is inherited by the generic residual intersections of an ideal generated by a regular sequence. This observation is quite surprising in light of the discussion so far since residual intersections do not preserve height unless they are links. 
\addtocontents{toc}{\protect\setcounter{tocdepth}{1}}
\subsection{Brief recall of residual intersections} \label{subsec: RI} The notion of residual intersections essentially goes back to Artin and Nagata \cite{Artin-Nagata}. Let $X$ and $Y$ be two irreducible closed subschemes of a Noetherian scheme $Z$ with $\codim_Z X \leq \codim_Z Y =s$ and $Y \nsubseteq X$, then $Y$ is called a residual intersection of $X$ if the number of equations needed to define $X\cup Y$ as a subscheme of $Z$ is the smallest possible, namely $s$. However, in order to include the case where $X$ and $Y$ are reducible with $X$ possibly containing some component of $Y$, the following algebraic definition is more suited:

\begin{defn}
Let $R$ be a polynomial ring and $I$ an $R$-ideal. Given an ideal $\mathfrak{a} \subsetneq I$ generated by $s$ elements, set $J = \mathfrak{a}:I$. If $\height(J) \geq s \geq \height(I)$, then $J$ is an \emph{$s$-residual intersection} of $I$ (or $R/J$ is an $s$-residual intersection of $R/I$). 

If furthermore $I_{\frakp}=\fraka_{\frakp}$ for all $\frakp \in V(I)$ with $\height{\frakp}\leq s$, then $J$ is called a \textit{geometric} $s$-residual intersection of $I$. 
\end{defn}

Residual intersections are natural generalizations of linkage: If $I$ is unmixed of height $g$, then the $g$-residual intersections of $I$ are precisely its links and the geometric $g$-residual intersections are precisely its the geometric links by Proposition \ref{prop:PS}. Furthermore, as in the case of links, we may define the `most general' residual intersections if the following technical requirement is met:

\begin{defn}
 We say that an ideal $I$ in a Noetherian ring satisfies the condition $G_s$ if $\mu(I_{\frakp})\leq \height (\frakp)$ for all prime ideals $\frakp$ containing $I$ such that $\height (\frakp)\leq s-1$ (where $\mu(-)$ denotes the minimal number of generators). We say that $I$ satisfies $G_\infty$ if it satisfies $G_s$ for each $s$. It is clear that this is equivalent to the condition that $\mu(I_{\frakp})\leq \height(\frakp)$ for each prime ideal $\frakp$ containing $I$.  
\end{defn}

\begin{defn}
Let $R$ be a polynomial ring and $I$ be an ideal of height $g>0$ satisfying $G_{s+1}$, where $s\geq g$. Choose any generating set $\underline{f}\colonequals f_1, \ldots, f_n$ of $I$ and let $Y$ be an $s \times n$ matrix of indeterminates. Let $\mathfrak{a}$ be the ideal generated by the entries of the matrix $Y[f_1\dots f_n]^T$. Then we set \[ 
\RI(s;I)=\RI(s;\underline{f})\colonequals \mathfrak{a}R[Y]:IR[Y]\] and call this ideal a \emph{generic $s$-residual intersection} of $I$. 
\end{defn}

Generic residual intersections are geometric residual intersections \cite[Theorem 3.3]{Huneke-Ulrich88}. As is the case with generic linkage, the generic residual intersections are essentially independent of the generating set for the ideal \cite[Lemma 2.2]{HunekeUlrichGenericRI} and specialize to any particuar $s$-residual intersection of the given ideal. 

\subsection{$F$-pure threshold of the generic residual intersections of a complete intersection} We layout the setup for this subsection: Let $\underline{\alpha}\colonequals \alpha_1, \ldots, \alpha_g$ be a homogeneous regular sequence, with each $\alpha_i$ having the same degree, in the polynomial ring $K[\underline{x}]$ such that the pair $(K[\underline{x}],\underline{\alpha})$ has $\mathbf{P}$. Fix an integer $s\geq g$ and let $Y\colonequals (Y_{i,j})$ be a matrix of indeterminates of size $s\times g$. Let 
\[\RI(s;\underline{\alpha}) \colonequals Y[\alpha_1, \cdots, \alpha_g]^TR[Y]: (\underline{\alpha})R[Y] \] denote the generic $s$-residual intersection of the ideal $(\underline{\alpha})$; set $M\colonequals Y[\alpha_1, \cdots, \alpha_g]^T$. Since an ideal generated by a regular sequence satisfies $G_\infty$, the residual intersection $\RI(s;\underline{\alpha})$ exists for each $s\geq g$. 

We now show that the generic residual intersections of complete intersections inherit $\mathbf{P}$. We need the generators of the residual intersection to do this, which were calculated in \cite[Example 3.4]{Huneke-Ulrich88} (see also \cite{Bruns-Kustin-Miller}).

\begin{thm} \ \label{thm:RIhasP}
If the pair $(R,\underline{\alpha})$ has $\mathbf{P}$, then $(R[Y],\RI(s;\underline{\alpha}))$ also has $\mathbf{P}$ for each $s \geq |\underline{\alpha}|$.    
\end{thm}

\begin{proof}
Since $\underline{\alpha}= \alpha_1,\ldots, \alpha_g$ is a regular sequence, by the structure theorem for the residual intersections of a complete intersection ring (\cite[Example 3.4]{Huneke-Ulrich88}), we have 
\[RI(s,\underline{\alpha}) = (\text{entries of the matrix $M$})R[Y]  +I_g(Y)R[Y]\]
is a homogeneous $R[Y]$-ideal of height $s$. We shall construct a term order $<_1$ in $R[Y]$ such that the pair $(R[Y],\underline{\beta})$ satisfies $\mathbf{P}$ with 
\[ \underline{\beta}\colonequals \{M_{1,1},\ldots, M_{g-1,1} \} \bigcup \{[i,i+1,\ldots,g-i+1]\; | \; 1 \leq i \leq s-g+1\}. \]
That is, $\underline{\beta}$ consists of the first $g-1$ entries of the matrix $M$ and all the $g$-minors of the matrix $Y\colonequals (Y_{i,j})$ with adjacent rows. Clearly, the set $\underline{\beta}$ has \[(g-1)+(s-g+1)=s=\height(\RI(s;\underline{\alpha}))\] 
elements, as required; hereafter, we call these elements $\beta_1, \ldots, \beta_{g-1},\beta_g,\ldots,\beta_s$ with the indexing according to the above display. We now construct the required term order.

Define the following variable order in $R[Y]$:
\begin{gather*}
Y_{s,g}>Y_{s,g-1}>\cdots > Y_{s,1}>\\
\vdots\\
Y_{g-1,g}>Y_{g-1,g-1}>\cdots > Y_{g-1,1}>\\
Y_{g-2,g-1}>Y_{g-2,g}>Y_{g-2,g-2}>\cdots > Y_{g-2,1}>\\
Y_{g-3,g-2}>Y_{g-3,g}>Y_{g-3,g-1}>\cdots > Y_{g-3,1}>\\
\vdots\\
Y_{1,2}>Y_{1,g}>Y_{1,g-1}>\cdots >Y_{1,1}>x_{i.j}.
\end{gather*}
That is, we begin the variable order with the last row of the matrix $Y$, move from right to left, and then vertically up along the rows of $Y$ for (the bottom) $s-g$ rows, thereafter we follow the same order for the remaining rows except that the $(i,i+1)$ entry appears first in the ordering of the $i$-th row. In addition, the ordering on the indeterminates $x_{i,j}$ is the one given in the property $\mathbf{P}$ for the pair $(R,\underline{\alpha})$. Given a monomial $u\in R[Y]$, we write $u_x\in R$ for the image of $u$ under the map of $R$-algebras from $R[Y]$ to $R$ sending each $Y_{i,j}$ to $1$. Consider the following term order $<_1$ on $R[Y]$: For monomials $u$ and $v$ in $R[Y]$, we define
\[u<_1 v \iff \begin{cases}
    u/u_x \mbox{ is lexicographically smaller than }v/v_x, \\
    u/u_x=v/v_x \mbox{ and }u_x<v_x.
\end{cases}\] 
Note that the initial terms of the first $g-1$ entries of the matrix $M$ are
\begin{align*}
 \ini_{<_1}(\beta_1) &= \ini_{<_1}(Y_{1,1}\alpha_1+Y_{1,2}\alpha_2+\cdots+Y_{1,g}\alpha_g) = Y_{1,2}\ini_<(\alpha_2),\\
  \ini_{<_1}(\beta_2) &= \ini_{<_1}(Y_{2,1}\alpha_1+Y_{2,2}\alpha_2+\cdots+Y_{2,g}\alpha_g) = Y_{2,3}\ini_<(\alpha_3),\\
 \vdots\\
  \ini_{<_1}(\beta_{g-1}) &= \ini_{<_1}(Y_{g-1,1}\alpha_1+Y_{g-1,2}\alpha_2+\cdots+Y_{g-1,g}\alpha_g)=Y_{g-1,g}\ini_<(\alpha_g).
\end{align*}
Further, by determinant expansion along the last row of $Y$, the initial terms of the remaining $s-g+1$ elements of $\underline{\beta}$,
\[\ini_{<_1}(\beta_i) = \ini_{<_1}([i,i+1,\cdots,g+i-1]) = Y_{i,1}Y_{i+1,2}\ldots Y_{g+i-1,g}, \quad 1 \leq i \leq s-g+1,\]
are just the products of the variables occurring in the main diagonal of the matrix $ [i,i+1,\cdots,g+i-1]$. Since $(R, \underline{\alpha})$ has $\mathbf{P}$ by hypothesis, we get that for the elements $\underline{\beta}= \beta_1,\ldots,\beta_s$ of $R[Y]$, each initial term is squarefree and each pair of initial terms is mutually coprime, as required for $(R[Y], \RI(s;\underline{\alpha}))$ to inherit $\mathbf{P}$.
\end{proof}

While the $F$-pure threshold of the generic link of an ideal is atleast as large as that of the ideal itself by \cite[Proposition 3.7]{Niu}, it is clearly bounded above by the height of the ideal due to Lemma \ref{lemma:fpt<height}. On the other hand:    
  
\begin{cor} \label{cor:FPTofRI}
If the pair $(R,\underline{\alpha})$ has $\mathbf{P}$, then, for each $s\geq |\underline{\alpha}|$, we have 
 \[\lct({\RI(s;\underline{\alpha}))} =\fpt({\RI(s;\underline{\alpha}))}=s \geq |\underline{\alpha}|=\fpt{(\underline{\alpha})}=\lct{(\underline{\alpha})}.\]
 \end{cor}

 \begin{proof}
Follows from Theorem \ref{thm:RIhasP} using Theorem \ref{theorem:main}. The outer equality is \cite{Hara-Yoshida}.      
 \end{proof}

\begin{rmk}\label{remark:answerKMN}
Corollary \ref{cor:FPTofRI} with $s=|\underline{\alpha}|$ answers \cite[Question 4.5(b)]{Kim-Miller-Niu} by identifying a class of complete intersection ideals---those possessing the property $\mathbf{P}$---for which the log canonical threshold is preserved under generic linkage.   

We point out that the log canonical threshold may \textit{not} be preserved under generic linkage (and therefore under generic residual intersections) for complete intersection ideals $I$ in a polynomial ring $R$ if $(R,I)$ does not have $\mathbf{P}$: Let $R=K[x_1,x_2,x_3]$ and $I=(x_1^2x_2,x_3^3)$. By \cite[Example 2.4(a)]{Kim-Miller-Niu} and the linear programming methods in \cite{Shibuta-Takagi}, we get
\[ \lct(L_1(I))=11/6>5/6=\lct(I).\]
\end{rmk} 



\section*{Acknowledgments}
I thank Manav Batavia, Linquan Ma, Yevgeniya Tarasova, Bernd Ulrich, and Matteo Varbaro for several insightful discussions. I am grateful to Samiksha Bidhuri for her support and encouragement. The author was partially supported by the NSF--FRG Grant DMS-1952366 and the AMS--Simons Travel Grant ASTG-23-284908.

\bibliographystyle{alpha}
\bibliography{main}

\end{document}